\documentclass[a4paper,12pt]{amsart}
\usepackage[utf8]{inputenc}
\usepackage{amsmath}
\usepackage{amsfonts}
\usepackage{amsthm}
\usepackage{amssymb}
\usepackage[mathscr]{euscript}
\usepackage{xcolor}
\usepackage{enumerate}
\usepackage{tikz}

\usepackage[
margin=1in,
marginpar=2cm,
includefoot,
footskip=30pt,
]{geometry}

\newtheorem{theorem}[equation]{Theorem}
\newtheorem{lemma}[equation]{Lemma}
\newtheorem{proposition}[equation]{Proposition}
\newtheorem{corollary}[equation]{Corollary}

\theoremstyle{definition}
\newtheorem{definition}[equation]{Definition}

\theoremstyle{remark}

\newtheorem{remark}[equation]{Remark}

\numberwithin{equation}{section}

\newcommand{\osf}{{\normalfont \textsf{X}}}

\newcommand{\lang}{\CL_{\osf}}

\newcommand{\ualgshift}{\TCA_R(\osf)}

\newcommand{\udalgshift}{\TCD_R(\osf)}

\newcommand{\alf}{\mathscr{A}}

\newcommand{\N}{\mathbb{N}}
\newcommand{\Z}{\mathbb{Z}}
\newcommand{\F}{\mathbb{F}}

\newcommand{\CA}{\mathcal{A}}

\newcommand{\CD}{\mathcal{D}}

\newcommand{\CF}{\mathcal{F}}

\newcommand{\CL}{\mathcal{L}}

\newcommand{\TCA}{\widetilde{\CA}}
\newcommand{\TCB}{\mathcal{U}}
\newcommand{\TCD}{\widetilde{\CD}}

\newcommand{\tauh}{\widehat{\tau}}

\newcommand{\nn}{\mathbb{N}}
\newcommand{\zn}{\mathbb{Z}}

\newcommand{\eword}{\omega}

\newcommand{\vecspan}{\operatorname{span}}

\title[The reduction theorem for algebras of subshifts]{The reduction theorem for algebras of one-sided subshifts over arbitrary alphabets}

\author[D. Bagio]{Dirceu Bagio}
\author[C. Gil Canto]{Cristóbal Gil Canto}
\author[D. Gonçalves]{Daniel Gonçalves}
\address[Dirceu Bagio, Daniel Gonçalves and Danilo Royer]{Departamento de Matem\'atica, Universidade Federal de Santa Catarina, 88040-970 Florian\'opolis SC, Brazil. }
\email{d.bagio@ufsc.br \\ daemig@gmail.com \\ daniloroyer@gmail.com}
\author[D. Royer]{Danilo Royer}
\address[Cristóbal Gil Canto]{Departamento de Matem\'atica Aplicada, E.T.S. Ingenier\'\i a Inform\'atica, Universidad de M\'alaga, Campus de Teatinos s/n. 29071 M\'alaga.   Spain.}
\email{cgilc@uma.es}

\begin{document}


\thanks{The third author was partially supported by Fapesc - Fundação de Amparo à Pesquisa e Inovação do Estado de Santa Catarina, Capes-Print, and CNPq - Conselho Nacional de Desenvolvimento Científico e Tecnológico - Brazil. The second author was supported by the Spanish Ministerio de Ciencia e Innovaci\'on   through project  PID2019-104236GB-I00/AEI/10.13039/\- 501100011033 and by the Junta de Andaluc\'{i}a  through project FQM-336,  all of them with FEDER funds. The second author gratefully acknowledges the hospitality during his research stay at Universidade Federal de Santa Catarina, Brazil.}

\begin{abstract} Let $R$ be a commutative unital ring, $\osf$ a subshift, and $\ualgshift$ the corresponding unital subshift algebra. We establish the reduction theorem for $\ualgshift$. As a consequence, we obtain a Cuntz-Krieger uniqueness theorem for $\ualgshift$ and we show that $\ualgshift$ is semiprimitive (resp. semiprime) whenever $R$ is a field (resp. a domain).
\end{abstract}

\maketitle

\section{Introduction}

The study of subshifts in symbolic dynamics is closely related to non-commutative algebras. In \cite{CuntzKrieger}, the famous Cuntz-Krieger algebra was associated with a subshift of finite type. A subshift of finite type over a finite alphabet can be viewed as the edge subshift associated with a graph and this has motivated the definition of graph algebras, both in the analytical and algebraic context. The connections of these algebras with symbolic dynamics are well documented. For example, graph $C^*$-algebras are used to characterize orbit equivalence of subshifts associated with directed graphs \cite{BCW}. Leavitt path algebras \cite{AbrAraMol, GroupoidLeavitt} and their classification theory can be used in connection with William's problem in symbolic dynamics \cite{CGGH, HazratClassificationLPA}. 

Symbolic dynamics also focus on the study of subshifts over infinite alphabets, which have practical applications \cite{KitchensBook, LindMarcus}. Recently, in \cite{BGGV}, the authors introduced algebras of one-sided subshifts over arbitrary alphabets, denoted by $\ualgshift$, with the aim of providing an algebraic description of conjugacy between subshifts over arbitrary, possibly infinite, alphabets. For finite alphabets, these conjugacy results can be seen as purely algebraic versions of the $C^*$-algebraic results in \cite{BrixCarlsen}. The authors of \cite{BGGV} proved that the subshift algebras are isomorphic to the Leavitt path algebras of certain labeled graphs \cite{BCGW21}, and they can also be realized as partial skew group rings and Steinberg algebras. Furthermore, they establish a two-way connection between the notion of subshift associated with an infinite countable alphabet given by \cite{OTW} and non-commutative algebras.

In this manuscript, we focus on the unital algebra associated with a subshift and aim to prove a key result in the context of non-commutative algebras: the Reduction Theorem. This theorem states that any nonzero element of a unital subshift algebra can be reduced, via left and right multiplication by appropriate elements of the algebra, to either a nonzero multiple of a projection or a specific sum of terms related to a minimal cycle without exit. A reduction theorem was originally given for Leavitt path algebras in \cite{socle}, which proved to be an extremely useful tool for characterizing ring-theoretic properties of Leavitt path algebras. Similar results can be found in \cite{relativecohn} and \cite{reduction} within the contexts of relative Cohn path algebras and ultragraph Leavitt path algebras, respectively. For subshift algebras, we will utilize our main theorem to establish a uniqueness theorem, akin to the Cuntz-Krieger uniqueness theorem in \cite{LPA_ring}. Additionally, we will prove that the algebras $\ualgshift$ are semiprimitive when the base ring $R$ is a field, and semiprime when the base ring $R$ has no zero divisors (see Corollary \ref{banana}).

The paper is organized as follows. In Section \ref{eleicoes}, we provide an overview of basic elements of symbolic dynamics and show an auxiliary result regarding periodicity properties of words generated by an arbitrary alphabet (Lemma~\ref{gcd1}). In Section~\ref{unital}, we recall the definition of the unital subshift algebra following \cite{BGGV}. For completeness, we include the realization of the unital subshift algebra as a partial skew group ring. We also prove an auxiliary result concerning subshift algebras (Lemma \ref{chucrute}), which is required for the proof of our main theorem. Section~\ref{reduction section} is dedicated to the main theorem of this manuscript, the Reduction Theorem (Theorem~\ref{reduction theorem}). Its proof is presented in several steps. We first introduce the concept of a (minimal) cycle without an exit in a subshift (Definitions~\ref{cyclexit} and \ref{mini} and Lemma~\ref{minimalversion01}), which plays a crucial role in the Reduction Theorem. In the final part of the paper, in Section~\ref{consequencesofreduction}, we extract fundamental results from the Reduction Theorem. We provide a Uniqueness Theorem for subshift algebras (Theorem~\ref{ckuniqueness}) under the assumption that every cycle in the subshift has an exit. Additionally, we demonstrate that a corner of a subshift algebra at a specific projection is isomorphic to the Laurent polynomial algebra (see Lemma~\ref{lem-Laurent-pol}). Furthermore, Proposition~\ref{ecologico} shows that the subshift algebra $\ualgshift$ is semiprimitive if $R$ is a field (particularly, it is semiprime since any semiprimitive ring is semiprime). Finally, we relax the condition on the ring $R$ and conclude that if $R$ has no zero divisors, then $\ualgshift$ is semiprime (Corollary~\ref{banana}).

\section{Symbolic Dynamics}\label{eleicoes}

In this section, we briefly recall the construction of subshifts over an arbitrary alphabet. In our work, $\nn=\{0,1,2,\ldots\}$ denotes the set of natural numbers. Let $\alf$ be a non-empty set, called an \emph{alphabet}, and let $\sigma$ be the \emph{one-sided shift map} on $\alf^\N$, that is, $\sigma$ is the map from $\alf^\N$ to $\alf^\N$ given by $\sigma(x)=y$, where $x=(x(n))_{n\in \N}$ and  $y=(x(n+1))_{n\in \N}$. Elements of $\alf^*:=\bigcup_{k=0}^\infty \alf^k$ are called \emph{blocks} or \emph{words}, and $\omega$ stands for the empty word. We set $\alf^+=\alf^*\setminus\{\eword\}$. Given $\alpha\in\alf^*\cup\alf^{\N}$, $|\alpha|$ denotes the length of $\alpha$ and, for $1\leq i,j\leq |\alpha|$, we define $\alpha(i,j):=\alpha(i)\cdots\alpha(j)$ if $i\leq j$, and $\alpha(i,j)=\eword$ if $i>j$. If $\beta\in\alf^*$, then $\beta\alpha$ denotes the usual concatenation. For a block $\alpha\in \alf^k$, $\alpha^n$ stands for the concatenation of $\alpha$ with itself $n$ times, and $\alpha^\infty$ denotes the infinite word $\alpha \alpha \ldots$. We say that $\beta$ is an \emph{initial segment} of $\alpha \in \alf^k$ if there exists $\alpha'\in \alf^*$ such that $\alpha=\beta \alpha'$, and $\beta$ is a \emph{final segment} of $\alpha \in \alf^k$ is there exists $\alpha''\in \alf^*$ such that $\alpha = \alpha'' \beta$. A subset $\osf \subseteq \alf^\N$ is \emph{invariant} for $\sigma$ if $\sigma (\osf)\subseteq \osf$. For an invariant subset $\osf \subseteq \alf^\N$, we define $\CL_n(\osf)$ as the set of all words of length $n$ that appear in some sequence of $\osf$, that is, 

$$\CL_n(\osf):=\{a(0) \ldots a(n-1)\in \alf^n:\ \exists \ x\in \osf \text{ s.t. } x(0)\ldots x(n-1)=a(0)\ldots a(n-1)\}.$$ 

Clearly, $\CL_n(\alf^\N)=\alf^n$ and we always have that $\CL_0(\osf)=\{\omega\}$.
The \emph{language} of $\osf$ is the set $\lang$, which consists of all finite words that appear in some sequence of $\osf$, that is,
$$\lang:=\bigcup_{n=0}^\infty\CL_n(\osf).$$

\noindent{\bf Notation.} We will write $\{\alpha_i\,:\,i\in I\}$ to denote a family (finite or infinite) of elements of $\lang$, whereas $\alpha(j)$ denotes the letter in the j-th position in the word $\alpha\in \lang$. 

\smallskip

Given $F\subseteq \alf^*$, we define the \emph{subshift} or \emph{shift space} $\osf_F\subseteq \alf^\N$ as the set of all sequences $x$ in $\alf^\N$ such that no word of $x$ belongs to $F$. Usually, the set $F$ will not play a role, so we will omit the subscript $F$ and write $\osf$, with the implication that $\osf=\osf_F$ for some $F$. 
\medbreak

At this point, we recall the definition of the key sets that will be used in the definition of the algebra associated with a subshift.

\begin{definition}\label{Arapaima}
Let $\osf$ be a subshift for an alphabet $\alf$. Given $\alpha,\beta\in \lang$, define \[C(\alpha,\beta):=\{\beta x\in\osf:\alpha x\in\osf\}.\] In particular, denote $C(\eword,\beta)$ by $Z_{\beta}$ and call it a \emph{cylinder set}. Moreover, we denote $C(\alpha,\eword)$ by $F_{\alpha}$ and call it a \emph{follower set}. Notice that $\osf=C(\eword,\eword)$.
\end{definition}

In the following lemma, which will be used later, $gcd(m,n)$ denotes the greatest common divisor of two natural numbers $m$ and $n$.

\begin{lemma}\label{gcd1}  Let $\alf$ be an alphabet and $\alpha$ and $\beta$ be two words in $\alf^+$.
\begin{enumerate}[\hspace{.3cm} \rm (1)]
\item If $\alpha\beta=\beta\alpha$, then there exist $n,m\in \N$ and  $c\in \alf^+$ such that $\alpha = c^n$, $\beta=c^m$, and $|c|=gcd(|\alpha|,|\beta|)$. 
\item If there are $p,q\in \N$ such that $\alpha^p=\beta^q$, then there exists $c\in\alf^+$ with $|c|=gcd(|\alpha|, |\beta|)$, and $m,n\in \N$ such that   $\alpha=c^m$ and $\beta=c^n$.

\item Let $c_1,...,c_k$ be non-empty words and $\{p_1,....,p_k\}\subseteq \N$ be such that $c_i^{p_i}=c_j^{p_j}$ for every $i,j\in \{1,...,k\}$. Then, there exist $c\in\alf^+$ and $\{q_1,...,q_k\} \subseteq \N $ such that $c_i=c^{q_i}$ for each $i\in \{1,...,k\}$.
\end{enumerate}

\end{lemma}

\begin{proof} 
We begin with the proof of (1). Without loss of generality, suppose that $|\alpha|>|\beta|$ (if $|\alpha|=|\beta|$ the result follows directly). From the commutativity hypothesis on $\alpha$ and $\beta$ we obtain that $\alpha \beta^n = \beta^n \alpha$ for every $n\in \N$. Then, there exist $n \geq 1$ and $0\leq r < |\beta|$ such that $\alpha = \beta^n \beta(1,r)$. If $r=0$ then we are done and so we assume that $0< r < |\beta|$. Substituting the expression we have obtained for $\alpha$ in the equation $\alpha \beta = \beta \alpha$, we conclude that $\beta(1,r)\beta = \beta \beta(1,r)$. If $|\beta(1,r)|$ divides $|\beta|$ then $\beta$ and $\alpha$ are powers of $\beta(1,r)$ and the result follows.  If $|\beta(1,r)|$ does not divide $|\beta|$ then, using that $\beta(1,r)\beta = \beta \beta(1,r)$ and proceeding as before, we obtain that there exist $n_1\geq 1$ and $0<r_1<r$ such that  $\beta = \beta(1,r)^{n_1}\beta(1,r_1)$. Substituting this expression for $\beta$ in the equality $\beta(1,r)\beta = \beta \beta(1,r)$, we obtain that $\beta(1,r)\beta(1,r_1)=\beta(1,r_1)\beta(1,r)$. Continuing as before, the process stops or we can write $\beta(1,r)= \beta(1,r_1)^{n_2} \beta(1,r_2)$, where $n_2\geq 1$ and $0\leq r_2 <r_1$. Clearly, this procedure eventually stops and we obtain $s\geq 1$ such that $\alpha$ and $\beta$ are powers of $\beta(1,s)$.

We now prove the second item. Assume, without loss of generality, that $|\beta|<|\alpha|$.
Suppose initially that $gcd(|\alpha|, |\beta|)=1$. From the Bezout identity, there exist $r,s\in \Z$ such that $r|\beta|=s|\alpha|+1$. Notice that if $r=0$ then $s=-1$ and $|\alpha|=1$, and if $s=0$ then  $r=1$ and $|\beta|=1$. In both cases, the result follows directly. We are left with the cases $r,s>0$ or $r,s<0$. Assume that $r,s>0$ (the other case is proved similarly). Let $x:=\alpha^p=\beta^q$ and define $y=x^\infty$. Then,  $y=\alpha\alpha\alpha\ldots = \beta\beta\beta\ldots$. Fix $j\in \{1,...,|\alpha|\}$.  Since $r|\beta|=s|\alpha|+1$, we have that $j r|\beta|=js|\alpha|+j$ and hence $y(jr|\beta|)=y(js|\alpha|+j)$. As $y=\beta^\infty$, we obtain that $y(jr|\beta|)=\beta(|\beta|)$, that is, $y(jr|\beta|)$ is equal to the last letter of $\beta$. Moreover, since $y=\alpha\alpha\alpha...$, we have that $y(js|\alpha|+j)$ is equal to $\alpha(j)$. Therefore, $\alpha(j)=\beta(|\beta|)$ for each $j\in \{1,...,|\alpha|\}$. We conclude that there exists $a\in\alf$ such that $\alpha=a^{m}$, for some $m\geq 1$, which implies that $\beta=a^n$ for some $n\geq 1$ (since $\alpha^p=\beta^q$). 

To finish the proof of the second item, let $k=gcd(|\alpha|, |\beta|)>1$. Write $\alpha=\gamma_1...\gamma_u$ and $\beta=\delta_1...\delta_v$, where $|\gamma_i|=|\delta_j|=k$, and notice that $gcd(u,v)=1$. In this case, $\gamma_i$ and $\delta_j$ can be seen as letters in the alphabet $\alf^k$, and the desired result follows from the previous paragraph of the proof.

Finally, we show the third item. We use an inductive argument on the number of words $c_i$. From the second item, the statement of the third item is true for two words. Suppose that it is also true for $k$ words, let $c_1,...,c_k,{c}_{k+1}$ be non-empty words and $\{p_1,...,p_k, p_{k+1}\}\subseteq \N$ be such that $c_i^{p_i}=c_j^{p_j}$, for each $i,j\in \{1,...,k+1\}$. Let $d\in \alf^+$ and $\{m_1,...,m_k\}\subseteq \N$ be such that $d^{m_i}=c_i$ for each $i\in \{1,...,k\}$. 
Then, $(d^{m_1})^{p_1}=c_1^{p_1}= c_{k+1}^{p_{k+1}}$. 
 From the second item there exist $c\in \alf^+$ and natural numbers $m,n$ such that $d=c^m$ and $c_{k+1}=c^n$. Therefore, for $i\in \{1,...,k\}$, we have that $c_i=d^{m_i}=(c^m)^{m_i}=c^{mm_i}$. Hence, the result follows by defining $q_i=mm_i$, for $i\in\{1,...,k\}$, and $q_{k+1}=n$. 
\end{proof}

\section{The unital subshift algebra}\label{unital}

Throughout the rest of the paper, $R$ denotes a commutative unital ring. In this section, we recall the definition of the unital subshift algebra, its description as a partial skew group ring, and prove properties regarding the multiplication between certain elements of the algebra.

\subsection{Definition and the partial skew group ring realization}\label{theoretic.pa}
Let $\TCB$ be the Boolean algebra of subsets of $\osf$ generated by all $C(\alpha,\beta)$ for $\alpha,\beta\in\lang$, that is, $\TCB$ is the collections of sets obtained from finite unions, finite intersections, and complements of the sets $C(\alpha,\beta)$. 


\begin{definition}\cite[Definition 3.3]{BGGV}\label{gelado} Let $\osf$ be a subshift.
The \emph{unital subshift algebra} $\ualgshift$ is the universal unital $R$-algebra  with generators $\{p_A: A\in\TCB\}$ and $\{s_a,s_a^*: a\in\alf\}$, subject to the relations:
\begin{enumerate}[(i)]
    \item $p_{\osf}=1$, $p_{A\cap B}=p_Ap_B$, $p_{A\cup B}=p_A+p_B-p_{A\cap B}$ and $p_{\emptyset}=0$, for every $A,B\in\TCB$;\vspace{.1cm}
    \item $s_as_a^*s_a=s_a$ and $s_a^*s_as_a^*=s_a^*$, for all $a\in\alf$;\vspace{.1cm}
    \item $s_{\beta}s^*_{\alpha}s_{\alpha}s^*_{\beta}=p_{C(\alpha,\beta)}$ for all $\alpha,\beta\in\lang$, where $s_{\eword}:=1$ and, for $\alpha=\alpha(1)\ldots\alpha(n)\in\lang$, $s_\alpha:=s_{\alpha(1)}\cdots s_{\alpha(n)}$ and $s_\alpha^*:=s_{\alpha(n)}^*\cdots s_{\alpha(1)}^*$.
\end{enumerate}
\end{definition}

According to \cite[Remark 3.4]{BGGV}, for $\alpha, \beta\in\lang$, we have that $s_\alpha^* s_\alpha = p_{C(\alpha,\omega)}=p_{F_\alpha}$ and $s_\beta s_\beta^* = p_{C(\omega,\beta)}= p_{Z_\beta}$.

The unital subshift algebra has a grading over the integers, which we recall below.

 \begin{proposition}\label{prop:grading} \cite[Proposition~3.9]{BGGV} Let $\osf$ be a subshift. The unital subshift algebra $\ualgshift$ is $\zn$-graded, with grading given by
     \[\ualgshift_n = \vecspan_R\{s_\alpha p_A s_\beta^* : \alpha,\beta \in \lang,\ A\in\TCB \ \mbox{and} \ |\alpha|-|\beta|=n\}.\]
 \end{proposition}

Next, we recall the description of the unital subshift algebra as a partial skew group ring, see \cite[$\S\,$5.1]{BGGV} for more details. This realization induces a grading of the algebra by the free group generated by the alphabet, which will be used throughout the rest of the paper (in particular to conclude that elements are non-zero). 

 Let $\CF(\osf,R)$ be the $R$-algebra of functions from $\osf$ to $R$ with pointwise operations, and $\udalgshift$ be the subalgebra of $\CF(\osf,R)$ generated by the characteristic functions of the sets $C(\alpha,\beta)$, where $\alpha,\beta\in\lang$. Let $\F$ be the free group generated by $\alf$, with the empty word $\eword$ as the identity of $\F$.

The construction begins with a partial action on the set level. Let $\tauh=\left(\{W_t\}_{t\in\F},\{\tauh_t\}_{t\in\F}  \right)$ be the partial action of $\F$ on $\osf$ such that, for every $\alpha,\beta\in\lang$ with $\beta\alpha^{-1}$ in the reduced form, $W_{\beta\alpha^{-1}} = C(\alpha,\beta)$ and
\[\tauh_{\alpha\beta^{-1}}(\beta x)=\alpha x,\]
for every $\beta x\in C(\alpha,\beta)$. Moreover, if $t\neq \alpha\beta^{-1}$ for every $\alpha,\beta\in\lang$, then $W_t=\emptyset$.

Next, an algebraic partial action is associated with $\tauh$, similarly to the dual action of a topological partial action. For $\alpha,\beta\in \lang$ such that $\alpha\beta^{-1}$ is in reduced form, let $1_{\alpha\beta^{-1}}$ denote the characteristic function of $W_{\alpha\beta^{-1}}=C(\beta,\alpha)$ and $D_{\alpha\beta^{-1}}$ the ideal of $\udalgshift$ generated by $1_{\alpha\beta^{-1}}$. Note that $D_{\alpha\beta^{-1}}$ is the ideal of functions in $\udalgshift$ that vanish on $C(\beta,\alpha)^c$ and has unit $1_{\alpha\beta^{-1}}$. Define $\tau_{\alpha\beta^{-1}}:D_{\beta\alpha^{-1}}\to D_{\alpha\beta^{-1}}$ by $\tau_{\alpha\beta^{-1}}(f)=f\circ\tauh_{\beta\alpha^{-1}}$, where $f\in D_{\beta\alpha^{-1}}$. Notice that $\tau_{\alpha\beta^{ -1}}(1_{\beta\alpha^{-1}})=1_{\alpha\beta^{-1}}$. If $t$ cannot be expressed in the form $\alpha\beta^{-1}$, we define $D_t=\{0\}$ and $\tau_t$ equals the zero function. Then, we have an algebraic partial action $\tau=\left( \{D_t\}_{t\in \F}, \{\tau_t\}_{t\in\F} \right)$  of $\F$ on $\udalgshift$.

The partial skew group ring associated with $\tau$ is defined as 
\[\udalgshift\rtimes_{\tau}\F= \bigoplus_{t\in\F} D_t=\left\{\sum_{} f_t\delta_t: t\in \F, f_t\in D_t \right\}, \]
where it is understood that $f_t$ is non-zero for finitely many terms and $\delta_t$ merely serves as a placeholder to remind us that $f_t\in D_t$. Multiplication is defined by 
\begin{equation}\label{eq:partial_multiplication}
    (f_s\delta_s)(g_t\delta_t) = \tau_s(\tau_s^{-1}(f_s)g_t )\delta_{st}.
\end{equation}

It is proved in \cite[Theorem 5.9]{BGGV} that $\udalgshift\rtimes_{\tau}\F$ is isomorphic to the unital subshift algebra, as stated below. 

\begin{theorem}\label{thm:set-theoretic-partial-action}
Let $\osf$ be a subshift. Then, $\ualgshift\cong\udalgshift\rtimes_{\tau}\F$ via an isomorphism $\Phi$ that sends $s_a$ to $1_a\delta_a$, $s^*_a$ to $1_{a^{-1}}\delta_{a^{-1}}$ for each $a \in \alf$, and $p_A$ to $1_A \delta_{\eword}$ for each $A\in \mathcal{U}$.
\end{theorem}

For future use, we also record the following lemma that is easy to prove.

\begin{lemma}\label{referee} Let $A$ be an algebra graded by a group $ G$, which has unit $e$. Suppose that $a_e\in A_e$ and $x=\sum_{g \in G} a_g \in A$. If $ a_g a_e \neq 0$ for some $g \in G$, then $x a_e \neq 0$.
\end{lemma}

We finish this subsection with another lemma regarding $ \udalgshift\rtimes_{\tau}\F$ that will be used in the proof of the Reduction Theorem (Theorem~\ref{reduction theorem}).

\begin{lemma}\label{claim1}
    Let $\tau=\left( \{D_t\}_{t\in \F}, \{\tau_t\}_{t\in\F} \right)$ be the partial action defined above, $t\in \F$, and $0\neq g\in D_t$. Then, there exists $a\in \alf$ such that $(g\delta_t)( 1_a\delta_a)\neq 0$.
\end{lemma}
\begin{proof}
Since $g$ is a non-zero element in $D_t$, by the definition of the partial action $\tau$, we have that $t=\alpha\beta^{-1}$ with $\alpha, \beta \in \lang$. 
Suppose first that $\beta =w$, that is, $t=\alpha$. Let $x\in W_t$ be such that $g(x)\neq 0$. Write $x=x(1)x(2)x(3)\ldots$ and let $a=x({|\alpha|+1})$. Then, 
$$\tau_t(\tau_{t^{-1}}(g)1_a))(x)=g(x)1_a\left(x(|\alpha|+1)x(|\alpha|+2)\ldots\right)=g(x).$$ Since $g(x)\neq 0$, we conclude that  
$\tau_t(\tau_{t^{-1}}(g)1_a))\neq 0$ and so 
$$g\delta_t 1_a\delta_a=\tau_t(\tau_{t^{-1}}(g)1_a)\delta_{ta}\neq 0.$$
Next, suppose that $\beta \neq w$, so that $t=\alpha\beta^{-1}$. Write $\beta=a\beta'$ with $a\in \lang$. In this case, notice that 
$$\tau_{t^{-1}}(g)1_a=\tau_{\beta \alpha^{-1}}(g)1_a=\tau_{\beta \alpha^{-1}}(g)$$ 
and therefore, since $g\neq 0$, we have that
$g\delta_t 1_a\delta_a=\tau_t(\tau_{t^{-1}}(g)1_a)\delta_{ta}=g\delta_{ta}\neq 0.$
\end{proof}




\subsection{Relative ranges and multiplicative properties.}

Recall from \cite{BGGV} that the \emph{relative range} of $(A,\alpha)$, with $A\in \TCB$ and $\alpha \in \lang$ is the set $$r(A,\alpha)=\{x\in \osf:\alpha x\in A\}.$$ In particular, notice that $r(\osf,\alpha)=\{x\in \osf:\alpha x\in \osf\}=F_{\alpha}$. 

The following auxiliary results regarding relative ranges will be useful in our work.

\begin{lemma}\label{relative range} Let $\osf$ be a subshift, $A\in \TCB$ be non-empty and $\alpha, \beta \in \lang$. Then,
\begin{enumerate}
[\hspace{.3cm} \rm (1)]
\item $r(F_\alpha, \beta)=F_{\alpha\beta}$, if $\alpha \beta \in \lang$. 
\item If $A\subseteq r(A,\alpha)$ then $\alpha^n \in \lang$ and $A\subseteq r(A,\alpha^n)$ for each $n\in \N$.
\end{enumerate}   
\end{lemma}
\begin{proof}
For the first item, note that
$$r(F_\alpha, \beta)=\{x\in \osf: \beta x\in F_\alpha\}=\{x\in \osf: \alpha \beta x\in \osf\}=F_{\alpha \beta}.$$
For the second item, let $x\in A$. By hypothesis, $x\in r(A,\alpha)$ and so $\alpha x\in A$. Since $A\subseteq r(A,\alpha)$ and $\alpha x\in A$, we obtain that $\alpha \alpha x \in A$. Thus, $x\in r(A, \alpha^2)$. Inductively, we obtain that $A\subseteq r(A,\alpha^n)$.
\end{proof}

Next, we explore properties related to the multiplication of certain elements in $\ualgshift$.

\begin{lemma}\label{chucrute} Let $\osf$ be a subshift, $a,b\in \alf$, and $\gamma,\alpha \in \lang$.
\begin{enumerate}[\hspace{.3cm} \rm (1)]
\item If $\beta:=b\gamma\in \lang$, then $s_\beta^*s_a=\delta_{b,a}s_\gamma^*p_{F_a}$.\vspace{.2cm}
\item If $A\in \TCB$, then $p_As_\alpha=s_\alpha p_{r(A,\alpha)}$ and $s_\alpha^*p_A=p_{r(A,\alpha)}s_\alpha^*$.
\item Let $x:=\sum\limits_{i=1}^m \lambda_i s_{\alpha_i} p_{A_i}$, where $\alpha_i\in \lang$, $\lambda_i \in R\setminus\{0\}$, and $A_i\in \TCB\setminus\{\emptyset\}$, for all $1\leq i\leq m$. Suppose that $x\neq 0$ and $\alpha_i\neq  \alpha_j$ for $i\neq j$. Then, there exist $A\in \TCB$ and a non-empty subset $J\subseteq \{1,...,m\}$ such that $A\subseteq A_j$, $A\subseteq F_{\alpha_j}$ for each $j\in J$, and  $$0\neq xp_A=\sum\limits_{i\in J} \lambda_i s_{\alpha_i}p_A.$$
\end{enumerate}
\end{lemma}

\begin{proof} The first item follows from \cite[Proposition 3.6 (i)]{BGGV}. 

We will show that $p_A s_\alpha=s_\alpha p_{r(A,\alpha)}$ by induction on the length of $\alpha$. If $\alpha=\omega$, then $r(A,\omega)=A$ and hence, since $s_\omega=1$, the result follows.
Suppose that $p_As_\beta=s_\beta p_{r(A,\beta)}$ for each $\beta \in \lang$ with $|\beta|=n$. Let $\alpha\in \lang$ with $|\alpha|=n+1$, and write $\alpha=\beta a$ with $a\in \alf$. 
Then, $$p_As_\alpha=p_A s_\beta s_a=s_\beta p_{r(A,\beta)}s_a=s_\beta s_a p_{r(r(A,\beta),a)}=s_\alpha p_{r(r(A,\beta),a)}.$$  
Since $$r(r(A,\beta),a)=\{x\in \osf: ax\in r(A,\beta)\}=\{x\in \osf: \beta ax\in A\}=r(A,\beta a)=r(A,\alpha),$$
it follows that $s_\alpha p_{r(r(A,\beta),a)}=s_\alpha p_{r(A,\alpha)}$. Hence, $p_As_\alpha=s_\alpha p_{r(A,\alpha)}$. Similarly, one proves that $s_\alpha^*p_A=p_{r(A,\alpha)}s_\alpha^*$.

In order to prove the third item, first we show that there exists $B \in \TCB$, $B\subseteq A_1$, such that $$0\neq xp_B=\sum\limits_i \lambda_i s_{\alpha_i}p_B.$$ Suppose, without loss of generality, that $\lambda_i s_{\alpha_i} p_{A_i}\neq 0$
 for each $i$. We observe that $0\neq \lambda_1 s_{\alpha_1} p_{A_1}=\lambda_1 s_{\alpha_1} p_{A_1}p_{A_1}$. From Theorem~\ref{thm:set-theoretic-partial-action} and Lemma \ref{referee}, we have that $xp_{A_1}\neq 0$. Hence, $$0\neq xp_{A_1}=\lambda_1 s_{\alpha_1}p_{A_1}+\big(\sum \limits_{i=2}^m\lambda_i s_{\alpha_i}p_{A_i}\big)p_{A_1}=\lambda_1 s_{\alpha_1}p_{A_1}+\sum \limits_{i=2}^m\lambda_i s_{\alpha_i}p_{A_i\cap A_1}.$$
  If $\sum \limits_{i=2}^m\lambda_i s_{\alpha_i}p_{A_i\cap A_1}=0$, we are done. So, suppose that $\sum \limits_{i=2}^m\lambda_i s_{\alpha_i}p_{A_i\cap A_1}\neq 0$. We may also assume, without loss of generality, that $\lambda_i s_{\alpha_i}p_{A_i\cap A_1}\neq 0$ for each $2\leq i \leq m$. Then, 
 $$0\neq \lambda_2s_{\alpha_2}p_{A_2\cap A_1}=\lambda_2s_{\alpha_2}p_{A_2\cap A_1}p_{A_2}$$ and therefore, by Theorem \ref{thm:set-theoretic-partial-action} and Lemma \ref{referee},
 $$0\neq xp_{A_1} p_{A_2}=\lambda_1 s_{{\alpha}_1}p_{A_1\cap A_2}+ \lambda_2 s_{\alpha_2} p_{A_1\cap A_2}+\sum\limits_{i=3}^m\lambda_i p_{A_i\cap (A_1\cap A_2)}.$$ Notice that it could occur that $\lambda_1 s_{\alpha_1}p_{A_1\cap A_2}=0$, but since $\lambda_2 s_{\alpha_2}p_{A_1\cap A_2}\neq 0 $ we still have that $xp_{A_1}p_{A_2}\neq 0$ (using Theorem~\ref{thm:set-theoretic-partial-action}). Applying this argument repeatedly,  we obtain $B\in \TCB$, $B\subseteq A_1$, such that $$0 \neq xp_B = \sum\limits_{j=1}^m \lambda_j s_{\alpha_j}p_B,$$ where $\lambda_j s_{\alpha_j} p_B \neq 0$ for each $1\leq j\leq m$ (observe that this $m$ may not coincide with the $m$ in the sum above, but we rewrite it as $m$ again for notational convenience). 
 
 We are now in a position to prove the statement of the third item. Taking into account that $s_{\alpha_1}p_B=s_{\alpha_1}s_{\alpha_1}^*s_{\alpha_1}p_B=s_{\alpha_1}p_{B\cap F_{{\alpha}_1}}$,
we have that (again by Theorem \ref{thm:set-theoretic-partial-action} and Lemma \ref{referee}),
$$0\neq xp_Bp_{F_{\alpha_1}} = \big(\sum\limits_{j=1}^m \lambda_j s_{\alpha_j}p_{B}\big)p_{F_{\alpha_1}}= \sum\limits_{j=1}^m \lambda_j s_{\alpha_j}p_{B\cap F_{\alpha_1}}.$$ 
Suppose, without loss of generality, that $$\lambda_j s_{\alpha_j}p_{B\cap F_{\alpha_1}}\neq 0$$ for each $2\leq j \leq m$. Since $$s_{\alpha_2}p_{B\cap F_{\alpha_1}}=s_{\alpha_2}s_{\alpha_2}^*s_{\alpha_2}p_{B\cap F_{\alpha_1}}=s_{\alpha_2}p_{(B\cap F_{\alpha_1}\cap F_{\alpha_2})},$$ we obtain from Theorem \ref{thm:set-theoretic-partial-action} and Lemma \ref{referee} that
$$0\neq xp_Bp_{F_{\alpha_1}}p_{F_{\alpha_2}} = \big(\sum\limits_{j=1}^m\lambda_j s_{\alpha_j}p_{(B\cap F_{\alpha_1})}\big)p_{F_{\alpha_2}}=\sum\limits_{j=1}^m\lambda_j s_{\alpha_j}p_{(B\cap F_{\alpha_1}\cap F_{\alpha_2})}.$$
Applying repeatedly this argument we get $A \in \TCB$ and a non-empty subset $J\subseteq \{1,...,m\}$ (which is the set of the elements $j\in \{1,...,m\}$ such that $\lambda_j s_{\alpha_j} p_A\neq 0$), such that 
$$0\neq xp_A=\sum\limits_{i\in J} \lambda_i s_{\alpha_i}p_A,$$ and moreover $A\subseteq A_j$ and $A\subseteq F_{\alpha_j}$ for each $j\in J$.
\end{proof}

\section{The Reduction Theorem}\label{reduction section}

In this section, we will prove the main result of this paper, which is Theorem~\ref{reduction theorem} and that we will call the \emph{Reduction Theorem}. A key concept related to this theorem is that of a cycle within a subshift, which we define below. 

\begin{definition}
    Let $\osf$ be a subshift, $\alpha\in \lang\setminus \{w\}$, and $\emptyset \neq A\in \TCB$. The pair $(A,\alpha)$ will be called  a \emph{cycle} if $A\subseteq r(A,\alpha)$.
\end{definition}

Observe that, if $(A,\alpha)$ is a cycle, then $A=\{\alpha^\infty\}$ or $A$ contains infinite elements. Indeed, for each $x\in A$ we have that $\alpha^n x \in A$ for all $n\in \N$ (since $A\subseteq r(A,\alpha^n)$, see Lemma~\ref{relative range}). Suppose that $A$ contains an element $y$ such that $y\neq \alpha^\infty$. Then, since $\alpha^n y\in A$ for each $n\in \N$ and  $\alpha^n y\neq \alpha^m y$ for $n\neq m$, we conclude that $A$ contains infinite elements, as desired. This discussion motivates the following definition.

\begin{definition}\label{cyclexit}
Let $\osf$ be a subshift. We say that a cycle $(A,\alpha)$ has an \emph{exit} if $A\neq \{\alpha^\infty\}$. Otherwise, in case $A=\{\alpha^\infty\}$, we say that $(A,\alpha)$ is a \emph{cycle without exit}.
\end{definition}

\begin{remark}
   The reader should compare the definition above with Definition~7.11 and Proposition~7.12 (and its proof) in  \cite{GillesDanie}.
\end{remark}

\begin{definition}\label{mini}
Let $\osf$ be a subshift. We say that a cycle without exit $(A,\alpha)$ is \emph{minimal} if there is no element $\beta\in \lang$, with $1\leq |\beta|<|\alpha|$, such that $\alpha=\beta^k$ for some $k\geq 2$.
\end{definition}

The next result characterizes when a cycle without exit is minimal.

\begin{lemma}\label{minimalversion01} Let $(A,\alpha)$ be a cycle without exit. Then $(A,\alpha)$ is minimal if, and only if, there is no element $
\beta \in \lang$, with $1\leq |\beta|<|\alpha|$, such that $\beta \alpha^\infty=\alpha^\infty$.
\end{lemma}

\begin{proof} Suppose that $(A,\alpha)$ is not minimal. Then, there exists $\beta\in \lang$ such that $1\leq |\beta|<|\alpha|$ and $\alpha=\beta^k$ for some $k\in \N$. This implies that $\alpha^\infty=\beta^\infty$ and hence $\beta \alpha^\infty=\beta \beta^\infty=\beta^\infty=\alpha^\infty$.

For the other implication, suppose that there exists $\beta\in \lang$, with $1\leq |\beta|<|\alpha|$, such that $\beta \alpha^\infty=\alpha^\infty$.  Let $\gamma \in \lang \setminus\{\omega\}$ be the element of minimal length such that $\gamma\alpha^\infty=\alpha^\infty$. Clearly, $1\leq |\gamma|<|\alpha|$. As $\gamma\alpha^\infty=\alpha^\infty$, we obtain that $\gamma^n \alpha^\infty=\alpha^\infty$ for each $n\in \N$, and consequently ${\gamma}^\infty=\alpha^\infty$. From this last equality, we get that $\alpha={\gamma}^m\delta$ for some $m\geq 2$ and $\delta\in \lang$ with $0\leq |\delta|<|\gamma|$. Suppose that $|\delta|>0$. Then, 
$${\gamma}^m\delta {\gamma}^\infty={\gamma}^m\delta \alpha^\infty=\alpha\alpha^\infty=\alpha^\infty={\gamma}^\infty$$ and therefore $\delta{\gamma}^\infty={\gamma}^\infty$. Hence, as ${\gamma}^\infty=\alpha^\infty$,  we conclude that $\delta\alpha^\infty=\alpha^\infty$, which is impossible by the minimality of $\gamma$. Therefore, $|\delta|=0$ and so $\alpha={\gamma}^m$, which means that $(A,\alpha)$ is not minimal.  
\end{proof}

\begin{remark}\label{minimal cycle} Let $(A,\alpha)$ be a cycle without exit and let $\beta\in \lang$ be the element of minimal length such that $\alpha=\beta^k$ for some $k\in \N$. Then, $(A, \beta)$ is a minimal cycle without exit. Indeed, if $(A,\beta)$ is not minimal then there exists $r\in \lang$, with $1\leq |r|<|\beta|$, such that $\beta=r^n$ for some $n\geq 2$. Thus $\alpha=r^{kn}$ which contradicts the choice of $\beta$.  \end{remark}

We now have all the ingredients necessary to prove the Reduction theorem. In particular, recall that there is an isomorphism $\Phi$ between $\ualgshift$ and $\udalgshift\rtimes_{\tau}\F$ (Theorem~\ref{thm:set-theoretic-partial-action}), which induces an $\F$-grading on $\udalgshift$. This grading will be used throughout the proof of the theorem, which we state below.

\begin{theorem}\label{reduction theorem} {\rm \textbf{(Reduction Theorem)}} Let $\osf$ be a subshift and $x\in \ualgshift$ be a non-zero element. Then, there exist $\mu,\nu\in \ualgshift$ such that  $\mu x\nu\neq 0$ and either:
\begin{enumerate}[\hspace{.3cm} \rm (1)]
    \item $\mu x\nu=\gamma p_D$ with $D\in \TCB$ and $\gamma \in R \setminus \{0\}$, or
    \item $\mu x\nu=\gamma_1 p_A +\sum\limits_{i=2}^k \gamma_i s_{\beta^{q_i}}p_A$, where $(A,\beta)$ is a minimal cycle without exit, $q_i\in \N\setminus \{0\}$, and $\gamma_i\in R \setminus \{0\}$ for $i=1,\ldots, k$.
\end{enumerate}

\end{theorem}

\begin{proof} 

Let $0\neq x\in \ualgshift$ and $\Phi$ be the isomorphism of Theorem~\ref{thm:set-theoretic-partial-action}.
We structure the proof in a few steps. The first one is the following.

\vspace{0.5pc}
\noindent {\bf Claim 1:} {\it There exists  $\mu\in \lang$ such that $xs_\mu\neq 0$ and $x s_\mu=\sum\limits\lambda_i s_{\alpha_i}p_{A_i}$, with $\lambda_i\in R$, $\alpha_i \in \lang\setminus \{w\}$, and $A_i\in \TCB$ for each $i$. Moreover, there exists $e\in \alf$ such that ${\alpha_i}(|\alpha_i|)=e$ for each $i$.}
\smallskip

By Proposition~\ref{prop:grading}, the element $x$ has the form $x=\sum\limits_i\gamma_i s_{a_i}p_{B_i}s_{b_i}^*$, with $\gamma_i\in R$, $a_i,b_i\in \lang$, and $B_i\in \TCB$ for each $i$. Regroup the terms of $x$ in the following way:
$$x=\sum\limits_{t\in \F}\sum\limits_{a_i{(b_i)}^{-1}=t}\gamma_i s_{a_i}p_{B_i}s_{b_i}^*.$$ For each $t\in \F$ such that $\sum\limits_{a_i{(b_i})^{-1}=t}\gamma_i s_{a_i}p_{B_i}s_{b_i}^*\neq 0$, let $h_t:=\sum\limits_{a_i{(b_i})^{-1}=t}\gamma_i s_{a_i}p_{B_i}s_{b_i}^*$. Thus, $x=\sum\limits_{t}h_t$, where each $h_t\neq 0$.  Notice that, for all $t \in \F$, there exists $0\neq g_t\in D_t$ such that $\Phi(h_t)=g_t\delta_t$. Therefore, 
$\Phi(x)=\Phi\big(\sum\limits_t h_t\big)=\sum\limits_t g_t\delta_t,$ with each $g_t\neq 0$. Fix a $t$, say $t_0$, and, following Lemma~\ref{claim1}, choose $e\in \alf$ such that $(g_{t_0}\delta_{t_0}) (1_e\delta_e)\neq 0$. This implies that $\left(\sum_t g_t \delta_t \right) 1_e \delta_e \neq 0$ and, since $\Phi$ is an isomorphism, we conclude that $0\neq \Phi^{-1}(\sum_t (g_t\delta_t) (1_e\delta_e))=xs_e$.
Thus, $$0\neq xs_e=\sum_i \gamma_i s_{a_i}p_{B_i}s_{b_i}^*s_e.$$ Now, for each $b_i$ with $|b_i|>0$, by Lemma~\ref{chucrute}(1), we have that either $s_{b_i}^*s_e=0$ or $b_i=ec_i$ for some $c_i \in \lang$ with $|c_i|< |b_i|$, so $s_{b_i}^*s_e=s_{c_i}^*p_{F_e}=p_{r(F_e,c_i)}s_{c_i}^*$. Hence, using Lemma \ref{chucrute} (2), we obtain that $xs_e$ has the form
$$xs_e=\sum\limits_i {\widetilde{\gamma_i}}s_{\widetilde{a_i}}p_{\widetilde{B_i}}s_{\widetilde{b_i}}^*,$$
 with $\widetilde{a_i}, \widetilde{b_i}\in \lang$ and 
 $|\widetilde{b_i}|<|b_i|$ for each $i$ with $|b_i|>0$. Applying the above argument repeatedly, we get Claim 1.

\smallskip
\noindent {\bf Claim 2:} {\it There exist $\beta,\mu \in \lang$ such that $s_\beta s_\beta^* x s_\mu\neq 0$ and 
$$s_\beta s_\beta^*xs_\mu=\sum\limits_i s_{\alpha_i}\sum\limits_j \lambda_j^ip_{B_j^i},$$ with $\alpha_i\in \lang$, $1\leq|\alpha_i|<|\alpha_{i+1}|$, and $\alpha_{i+1}=\alpha_i\beta_i$, for some $\beta_i\in \lang$, $\lambda_j^i\in R$ and $B_j^i\in \TCB$. Moreover, $\alpha_i(|\alpha_i|)=\alpha_j(|\alpha_j|)$ for each $i, j$.}
\smallskip

Let $\mu\in \lang$ be as in Claim 1. Then, $xs_\mu=\sum\limits\lambda_i s_{\alpha_i}p_{A_i}$,  with $\lambda_i\in R$, $\alpha_i\in \lang\setminus \{w\}$, and $A_i\in \TCB$ for each $i$. Reorganize this last sum and write it in the form   
$$xs_\mu=\sum\limits_i s_{\alpha_i}\sum\limits_j \lambda_j^ip_{A_j^i},$$ with $\alpha_i\neq \alpha_j$ for each $i\neq j$ and $0\neq s_{\alpha_i}\sum\limits_j \lambda_j^ip_{A_j^i}$ for each $i$.
 Let $\beta$ be one of the $\alpha_i$ of maximum length, say, $\beta=\alpha_k$. Since $s_\beta s_\beta^*s_\beta=s_\beta$, we obtain that $$s_\beta s_\beta^*s_{\alpha_k}\sum\limits_j \lambda_j^kp_{A_j^k}=s_{\alpha_k}\sum\limits_j \lambda_j^kp_{A_j^k} \neq 0,$$ and therefore $s_\beta s_\beta^*xs_\mu\neq 0$. 
 Fix an index $i\neq k$. If $\alpha_i$ is not an initial segment of $\beta$, that is  $\alpha_i\neq \beta(1)\cdots\beta(|\alpha_i|)$, then we have that $s_\beta^*s_{\alpha_i}=0$ and hence $s_\beta s_\beta^* s_{\alpha_i}=0$. If $\alpha_i$ is an initial segment of $\beta$, that is $\beta=\alpha_i\gamma_i$ with $\gamma_i\in \lang$, then $$s_\beta s_\beta^*s_{\alpha_i}=s_{\alpha_i}s_{\gamma_i} s_{\gamma_i}^*s_{\alpha_i}^*s_{\alpha_i}\stackrel{(\star)}{=}s_{\alpha_i}s_{\alpha_i}^*s_{\alpha_i}s_{\gamma_i} s_{\gamma_i}^*=s_{\alpha_i} p_{Z_{\gamma_i}}.$$
Observe that in $(\star)$ we used that $s_{\gamma_i} s_{\gamma_i}^*$ and $s_{\alpha_i}^*s_{\alpha_i}$ commute.
Now, define $B_j^i:=A_j^i\cap Z_{\gamma_i}$, and notice that 
 
 $$s_\beta s_\beta^* s_{\alpha_i}\sum\limits_j \lambda_j^i p_{A_j^i}=s_{\alpha_i} \sum\limits_j \lambda_j^i p_{B_j^i}.$$
Let $F$ be the set of all indices $i$ such that $$s_\beta s_\beta^* s_{\alpha_i}\sum\limits_j \lambda_j^ip_{B_j^i}\neq 0.$$ 
Then  $$s_\beta s_\beta^*xs_\mu=\sum\limits_{i\in F} s_{\alpha_i}\sum\limits_j \lambda_j^ip_{B_j^i}.$$
Since the only possible $i\in F$ are those such that $\alpha_i$ is an initial segment of $\beta$, we are done.

\smallskip
\noindent {\bf Claim 3:} {\it There exist $y,z\in \ualgshift$ and a non-empty index set $J$ such that $yxz\neq 0$ and 
$$yxz=\sum\limits_{i\in J}\gamma_i s_{\alpha_i}p_A,$$ where $\gamma_i \in R$, $\alpha_i \in \lang$ with $1\leq |\alpha_i|<|\alpha_{i+1}|$,  $\alpha_{i+1}=\alpha_i\beta_i$ for some $\beta_i\in \lang$ for each $i \in J$, and $A\in \TCB$ with $A\subseteq F_{\alpha_i}$ for each $i\in J$. Moreover, $\alpha_i(|\alpha_i|)=\alpha_j(|\alpha_j|)$ for each $i, j$.}
\smallskip

Let $\mu, \beta$ be as in Claim 2 and take $y:=s_\beta s_\beta^*$. Then,
$$yxs_\mu=\sum\limits_{i=1}^n s_{\alpha_i}\sum\limits_j \lambda_j^ip_{B_j^i},$$  where $0\neq s_{\alpha_i} \sum\limits_j \lambda_j^ip_{B_j^i}$ for every $i$. 
By \cite[Lemma 3.5]{BCGW21}, the sum $\sum \limits_j \lambda_j^1p_{B_j^1}$ can be written in the form $\sum \limits_j \lambda_j^1p_{B_j^i}=\sum\limits_k {\gamma_k^1}p_{C_k^1}$, where $C_k^1 \in \TCB$ are pairwise disjoint and $\gamma_k^1\in R$.
Since $s_{\alpha_1}\sum\limits_j \lambda_j^1p_{B_j^1}\neq 0$, we have that $s_{\alpha_1}\gamma_k^1p_{C_k^1}\neq 0$ for some $k$. Multiplying $yxs_\mu$ on the right side by this $p_{C_k^1}$ we get (from Theorem \ref{thm:set-theoretic-partial-action} and Lemma \ref{referee}) that,
$$0\neq yxs_\mu p_{C_k^1}=\gamma_k^1s_{\alpha_1}p_{C_k^1}+\sum\limits_{i=2}^n s_{\alpha_i}\sum\limits_j \lambda_j^ip_{B_j^i\cap C_k^1}.$$
Applying this argument repeatedly (proceeding in the next step with $\sum\limits_{i=2}^n s_{\alpha_i}\sum\limits_j \lambda_j^ip_{B_j^i\cap C_k^1}$) we obtain an element $C\in \TCB$ such that 
$$0\neq y xs_\mu p_C=\sum\limits_{i=1}^m\gamma_is_{\alpha_i}p_{C_i},$$ where $1\leq |\alpha_i|<|\alpha_{i+1}|$, $\alpha_i$ is an initial segment of $\alpha_{i+1}$ for each $i$, and $\alpha_i(|\alpha_i|)=\alpha_j(|\alpha_j|)$ for each $i,j$. In the sum $\sum\limits_{i=1}^m\gamma_is_{\alpha_i}p_{C_i}$ we may suppose, without loss of generality, that $\gamma_is_{\alpha_i}p_{C_i}\neq 0$ for each $i$.
Using Lemma~\ref{chucrute}(3), applied on the element $y x s_\mu p_C$, we obtain an element $A\in \TCB$ and a non-empty index set $J$ such that $$0\neq yxs_\mu p_Cp_A=\sum\limits_{i\in J}\gamma_i s_{\alpha_i}p_A.$$ Recall that $A \subseteq F_{\alpha_i}$ for all $i \in J$. Defining $z:=s_\mu p_C p_A$, we get the desired result.

\smallskip
\noindent {\bf Claim 4:} {\it There exist $y',z'\in \ualgshift$ such that $y'xz'\neq 0$ and either $y'xz'=\gamma_1p_D$ for some $\gamma_1 \in R\setminus\{0\}$, $D\in \TCB$, or 
$$y'xz'=\gamma_1 p_B+\sum\limits_{i=2}^k\gamma_i s_{b_i} p_B,$$ where $\gamma_i \in R$, $b_i\in \lang$ with $1\leq |b_i|<|b_{i+1}|$, $b_{i+1}=b_i\beta_i$ for some $\beta_i \in \lang$ for each $i$, and $B\in \TCB$ with $B\subseteq F_{b_i}\cap Z_{b_i}$, for each $i$. Moreover, $b_i(|b_i|)=b_j(|b_j|)$ for each $i, j$.}
\smallskip

Let $y,z$ be as in Claim 3 and write $yxz=\sum\limits_{i=1}^m\gamma_i s_{\alpha_i}p_A$. As a consequence of the choice of the index set $J$ in Claim 3, we have that $\gamma_is_{\alpha_i}p_{A}\neq 0$ for each $i\in \{1,...,m\}$. By Claim 3, for each $i\in \{2,...,m\}$, there exists $b_i\in \lang$ such that $\alpha_i=\alpha_1b_i$. Moreover, $A\subseteq F_{\alpha_i}$ for each $i\in \{1,...,m\}$ and, in particular, $A\subseteq F_{b_i}$ for each $i\in \{2,...,m\}$. Notice that $0\neq \gamma_1s_{\alpha_1}p_A=\gamma_1s_{\alpha_1}s_{\alpha_1}^*s_{\alpha_1}p_A$, which implies that $0\neq \gamma_1s_{\alpha_1}^*s_{\alpha_1}p_A=\gamma_1 p_{F_{\alpha_1}}p_A=\gamma_1 p_A$, where the last equality holds since $A\subseteq F_{\alpha_1}$. Therefore, multiplying $yxz$ on the left side by $s_{\alpha_1}^*$ and taking into account Theorem \ref{thm:set-theoretic-partial-action} and Lemma \ref{referee}, we get that
\begin{equation}\label{eq2} 0 \neq s_{\alpha_1}^*yxz=\gamma_1 p_A + \sum\limits_{i=2}^m \gamma_i p_{F_{\alpha_1}}s_{b_i}p_A.
\end{equation}
From Lemma~\ref{relative range} and Lemma~\ref{chucrute}(2), we obtain that
$$p_{F_{\alpha_1}}s_{b_i}p_A=s_{b_i}p_{r(F_{\alpha_1}, b_i)}p_A=s_{b_i}p_{F_{\alpha_1b_i}}p_A,$$
for each $i\in \{2,...,m\}$.
Since $A\subseteq F_{\alpha_i}= F_{\alpha_1b_i}$, we conclude that $s_{_ib}p_{F_{\alpha_1b_i}}p_A=s_{b_i}p_A.$ Therefore, (\ref{eq2}) has the form 
\begin{equation}\label{eq3}
0\neq s_{\alpha_1}^*yxz=\gamma_1 p_A+\sum\limits_{i=2}^m\gamma_i s_{b_i}p_A.
\end{equation}

Suppose that $A\cap Z_{b_i}=\emptyset$ for each $i\in \{2,...,m\}$. Then, 
$$p_As_{b_i}=p_A s_{b_i}s_{b_i}^*s_{b_i}=p_A p_{Z_{b_i}}s_{b_i}=0,$$ for every $i\in \{2,...,m\}$. Hence, multiplying (\ref{eq3}) on the left side by $p_A$ we obtain that
$$0\neq p_A s_{\alpha_1}^*yxz=\gamma_1 p_A,$$
and Claim 4 follows by defining $y':=p_As_{\alpha_1}y$, $z':=x$ and $D:=A$.

We are left with the case where $A\cap Z_{b_i}\neq \emptyset$ for some $i\in \{2,...,m\}$. Let $k$ be the greatest element of the set $ \{2,...,m\}$ such that $A\cap Z_{b_k}\neq \emptyset$. Then, multiplying (\ref{eq3}) on the right side by $p_{Z_{b_k}}$ we get (taking into account Theorem \ref{thm:set-theoretic-partial-action} and Lemma \ref{referee}):
$$0\neq s_{\alpha_1}^*yxzp_{Z_{b_k}}=\gamma_1 p_{(A\cap Z_{b_k})}+\sum\limits_{i=2}^k \gamma_i s_{b_i}p_{(A\cap Z_{b_k})}.$$ Define $y':=s_{\alpha_1}^*y$, $z':=zp_{Z_{b_k}}$ and $B:=A\cap Z_{b_k}$. Since $B\subseteq Z_{b_k}$ and $b_i$ is an initial segment of $b_k$, for each $i\in \{2,...,k\}$, we have that $B\subseteq Z_{b_i}$ for each $i\in \{2,...,k\}$. Moreover, for each $i$, we have that $A\subseteq F_{\alpha_i}$ and $B\subseteq F_{b_i}$ (as 
$\alpha_i=\alpha_1b_i$ and $B\subseteq A$). Thus, Claim 4 is proved.

\smallskip
\noindent {\bf Claim 5: }{\it Suppose that $$ 0 \neq \; x=\gamma_1 p_B+\sum\limits_{i=2}^k\gamma_i s_{b_i}p_B,$$ where $\gamma_i \in R \setminus \{0\}$, $b_i\in \lang$ with $1\leq |b_i|<|b_{i+1}|$, $b_{i+1}=b_i\beta_i$ for some $\beta_i\in \lang$ for each $i$, $B\in \TCB$ with $B\subseteq F_{b_i}\cap Z_{b_i}$ for each $i$, and $b_i(|b_i|)=b_j(|b_j|)$ for each $i, j$. Then, there exist $y',z'\in \ualgshift$ such that $0\neq y'xz'$ and either 
$y'xz'=\lambda_1 p_C$, or $$y'xz'=\lambda_1 p_C+\sum\limits_{i=2}^m\lambda_i s_{c_i}p_C,$$
where $\lambda_i \in R \setminus \{0\}$, $c_i\in \lang$ with $1\leq |c_i|<|c_{i+1}|$ and $c_{i+1}=c_i\alpha_i$ for some $\alpha_i\in \lang$ for each $i$, $C\in \TCB$ with $C\subseteq F_{c_i}\cap Z_{c_i}$ for each $i$, $c_i(|c_i|)=c_j(|c_j|)$ for each $i, j$, and $C\subseteq r(C,c_2)$.
}
\smallskip

If $B\subseteq r(B, b_2)$ then we take $C=B$ and the result follows. Therefore, we suppose that $B\nsubseteq r(B, b_2)$, so that $D=B\setminus r(B, b_2)\neq\emptyset$. For each $i\in \{2,...,k\}$, we write $b_i=b_2d_i$. Then, multiplying $x$ on the left side by $p_Ds_{b_2}^*$, we get 
\begin{align*}
p_D s_{b_2}^*x &=\gamma_1 p_D s_{b_2}^*p_B+\gamma_2 p_D s_{b_2}^* s_{b_2}p_B+\sum\limits_{i=3}^k \gamma_i p_D s_{b_2}^*s_{b_2} s_{d_i}p_B\\
&=\gamma_1 p_D p_{r(B,b_2)}s_{b_2}^*+\gamma_2 p_D + \sum\limits_{i=3}^k \gamma_i p_D p_{F_{b_2}}s_{d_i}p_B\\
&\stackrel{(\star)}{=}\gamma_2p_D+\sum\limits_{i=3}^k \gamma_i s_{d_i}p_{r(D,d_i)}p_B,
\end{align*}
 where  $(\star)$ follows since $D\cap r(B,b_2)=\emptyset$ and $D\subseteq B\subseteq F_{b_2}$.
Also, by Lemma~\ref{chucrute}(3), there exist an element $E\in \TCB$ and a non-empty index set $J$ such that 
$$0\neq p_D s_{b_2}^* xp_E=\gamma_2 p_E+\sum \limits_{i\in J} \gamma_i s_{d_i}p_E.$$
Repeating this process (in the next step we look at the set $E\setminus r(E,d_3)$), we get Claim~5.

\smallskip
\noindent {\bf Claim 6: }{\it Suppose that 

$$0 \neq \; x=\lambda_1 p_C+ \lambda_2 s_\alpha p_C+\sum\limits_{i=3}^m\lambda_i s_{c_i}p_C,$$
where $\lambda_i \in R \setminus\{0\}$, $c_i\in \lang$ with $1\leq |\alpha|<|c_i|<|c_{i+1}|$, $\alpha$ is an initial segment of each $c_i$, $c_i$ is an initial segment of $c_{i+1}$ for each $i$, $C\in \TCB$ with $C\subseteq F_{\alpha}\cap F_{c_i}$, $C\subseteq Z_\alpha\cap Z_{c_i}$ for each $i$, and $C\subseteq r(C,\alpha)$. Then, there exist a non-empty index set $J$ and $A\in \TCB$ such that
$$0\neq s_\alpha s_\alpha^* (s_\alpha^*)^n x s_\alpha^n p_Cp_A=\lambda_1 p_A+\lambda_2 s_\alpha  p_A+\sum\limits_{i\in J} \lambda_i s_{(\beta_i  \alpha^{k_i})}p_A,$$
where $0\leq |\beta_i|<|\alpha|$, $\beta_i$ is an initial and final segment of $\alpha$, $k_i \in \N$ and $A \subseteq F_{\alpha} \cap F_{\beta_i  \alpha^{k_i}}$ for each $i \in J$}.
\smallskip

We begin by describing the effect of the multiplication of the monomials that form $x$ by appropriate elements on the right and on the left. In the end, we will collect all the results obtained. 

Let $n\in \N$ be such that $|\alpha^{n-1}|<|c_m|\leq |\alpha^n|$.

Notice that $$(s_\alpha^*)^np_C s_\alpha^n=s_{\alpha^n}^* s_{\alpha^n}p_{r(C,\alpha^n)}=p_{F_{\alpha^n}}p_{r(C,\alpha^n)}=p_{r(C,\alpha^n)}.$$ 
As $C\subseteq r(C,\alpha)$, the second item of Lemma~\ref{relative range} implies that $C\subseteq r(C,\alpha^n)$. Therefore, using the $\F$-grading (Theorem \ref{thm:set-theoretic-partial-action} and Lemma \ref{referee}), we obtain that $(s_{\alpha}^*)^n xs_\alpha^n p_C \neq 0$.

Next, observe that 
$$(s_\alpha^*)^ns_\alpha p_C s_\alpha^n p_C=(s_\alpha^*)^ns_\alpha^{n+1}p_{r(C,\alpha^n)}p_C=p_{F_{\alpha^n}}s_{\alpha}p_C=s_\alpha p_{F_{\alpha^{n+1}}}p_C=s_\alpha p_C,$$
where the last equality follows from the second item of Lemma~\ref{relative range}.

Proceeding, we compute $(s_\alpha ^*)^ns_{c_i}p_C s_{\alpha}^n$ for $i\geq 3$. Note that $(s_\alpha ^*)^ns_{c_i}=0$, unless $c_i$ has the form $c_i=\alpha^k t_i$, with $n \geq k$, $0\leq |t_i|<|\alpha|$, and $t_i$ is an initial segment of $\alpha$ for each $i$. 
If $|t_i|=0$ then $c_i=\alpha^k$ and we have that
\begin{align*} (s_\alpha^*)^ns_\alpha^kp_C s_\alpha^n &=(s_\alpha^*)^{n-k}p_{F_{\alpha}^k}p_Cs_{\alpha^n}=(s_\alpha^*)^{n-k}p_Cs_{\alpha^n}\\
& =(s_\alpha^*)^{n-k}s_{\alpha^n}p_{r(C,\alpha^n)} = p_{F_{\alpha^{n-k}}}s_{\alpha}^kp_{r(C,\alpha^n)} \\
&= s_\alpha^k p_{F_{\alpha^n}}p_{r(C,\alpha^n)}=s_\alpha^k p_{r(C,\alpha^n)}.
\end{align*}
Suppose now that $|t_i|>0$. Then we have that
\begin{align*}
 (s_\alpha^*)^ns_{c_i}p_Cs_\alpha^n &=(s_\alpha^*)^{n-k}(s_\alpha^*)^k s_\alpha^k s_{t_i} p_C s_\alpha^n=(s_\alpha^*)^{n-k} p_{F_{\alpha^k}} s_{t_i} p_C s_\alpha^n\\
 &=(s_\alpha^*)^{n-k} s_{t_i} p_{F_{\alpha^k t}} p_C s_\alpha^n=(s_\alpha^*)^{n-k} s_{t_i} p_C s_\alpha^n\\
 &=(s_\alpha^*)^{n-k} s_{t_i}  s_\alpha^n p_{r(C, \alpha^n)}
\end{align*}
Notice that $(s_{\alpha}^*)^{n-k} s_{t_i} s_\alpha^{n-k}=0$, unless $\alpha^{n-k}$ is an initial segment of $t_i\alpha^{n-k}$, in which case we have that $$t_i \alpha^{n-k}=\alpha^{n-k} \beta,$$ where $\beta=\alpha(j_i)\cdots\alpha(|\alpha|)$. Thus,
$$(s_{\alpha}^*)^{n-k} s_{t_i} s_\alpha^{n-k}=
(s_{\alpha}^*)^{n-k} s_\alpha^{n-k}s_{\beta}=
 p_{F_{\alpha^{n-k}}} s_{\beta},$$
and consequently
\begin{align*}
(s_\alpha^*)^ns_{c_i}p_Cs_\alpha^n p_C&=p_{F_{\alpha^{n-k}}}s_{\beta}s_{\alpha}^kp_{r(C,\alpha^n)}p_C  \\
&=p_{F_{\alpha^{n-k}}}s_{\beta}s_{\alpha}^k p_C\\
&=s_{\beta \alpha^k}p_{F_{(\alpha^{n-k}\beta \alpha^k)}}p_C\\
&=s_{\beta \alpha^k}p_{F_{(t_i \alpha^n)}}p_C.
\end{align*}
We observe that $|\beta|=|t_i|<|\alpha|$, $\beta$ is a final segment of $\alpha$, and $t_i$ is an initial segment of $\alpha$.  
Therefore, $s_\alpha^* s_\beta=0$ unless $t_i=\beta$. In this case, we write $\alpha=\beta r$ and get 
\begin{align*}
s_\alpha s_\alpha^* s_{\beta \alpha^k}&=s_\beta s_r s_r^* s_\beta^* s_\beta s_{\alpha^k}=s_\beta p_{Z_r} p_{F_{\beta}} s_{\alpha^k}=s_\beta s_{\alpha^k}p_{r(Z_r,\alpha^k)}p_{F_{\beta\alpha^k}}.
\end{align*}
Hence, multiplying $(s_\alpha^*)^ns_{c_i}p_Cs_\alpha^n p_C$ on the left side by $s_\alpha s_\alpha^*$ we obtain that
$$s_\alpha s_\alpha^*(s_\alpha^*)^ns_{c_i}p_Cs_\alpha^n p_C= s_{\beta\alpha^k}p_{r(Z_r,\alpha^k)}p_{F_{(\beta a^n)}}p_C$$ where $\beta$ is an initial and final segment of $\alpha$, and we have used that $F_{\beta \alpha^n}\subseteq F_{\beta \alpha^k}$ (since $n\geq k$ and $\beta$ is a final segment of $\alpha$). 

Since $k$ (resp. $\beta$) above depends on $i$, we denote the $k$ (resp. $\beta$) corresponding to the monomial with $s_{c_i}$ by $k_i$ (resp. $\beta_i$).

Now, multiplying $x$ on the left side by $s_\alpha s_\alpha^* (s_\alpha^*)^n$, on the right side by $s_\alpha^n p_C$, and using the $\F$-grading, we get that
$$0\neq y= s_\alpha s_\alpha^* (s_\alpha^*)^n x s_\alpha^n p_C=\lambda_1 p_C+\lambda_2s_\alpha p_C+\sum\limits_{i \geq 3} \gamma_i s_{(\beta_i \alpha^{k_i})}p_{r(Z_{r_i}, \alpha^{k_i})}p_{F_{(\beta_i \alpha^n)}}p_C.$$
To conclude, we use the third item of Lemma~\ref{chucrute} to obtain $A\in \TCB$ and a non-empty index set $J$ such that 
$$0\neq y p_A=\lambda_1 p_A+\lambda_2 s_\alpha p_A+ \sum\limits_{i \in J} \gamma_i s_{(\beta_i \alpha^{k_i})}p_A,$$ with $A \subseteq F_{\alpha} \cap F_{\beta_i  \alpha^{k_i}}$ for each $i \in J$ and Claim 6 is proved.

\smallskip
\noindent {\bf Claim 7: }
{\it Let $x \in \ualgshift$ be a non-zero element. There exist $x',y'\in \ualgshift$ such that $0\neq x'xy'$ and either $x'xy'=\gamma p_D$ with $\gamma \in R$ and $D\in \TCB$, or 
$$x'xy'=\gamma_1 p_A +\sum\limits_{i=2}^k \gamma_i s_{c^{p_i}}p_A,$$
where $c\in \lang$,  $0\neq \gamma_i\in R$ for each $i\in \{1,...,k\}$, $1\leq p_i <p_{i+1}$, $A\subseteq F_{c^{p_i}}\cap Z_{c^{p_i}}$, and  $A\subseteq r(A,c^{p_2})$.}  
\smallskip

Using Claim 4, and applying repeatedly Claim~5 and Claim~6 (if necessary), we obtain $x',y'\in \ualgshift$ such that $0\neq x'xy'$ and either $x'xy'=\gamma p_D$ with $\gamma \in R$ and $D\in \TCB$ (in which case we are done), or  
$$x'xy'=\gamma_1 p_A +\gamma_2 s_\alpha p_A +\sum\limits_{i=3}^k \gamma_i s_{\beta_i \alpha^{k_i}}p_A,$$ where $\gamma_i\in R$ and $\gamma_1, \gamma_2\neq 0$, $0\leq |\beta_i|<|\alpha|$, $\beta_i$ is an initial and final segment of $\alpha$, $k_i> 0$, $A\subseteq F_{\alpha}\cap  F_{\beta_i\alpha^{k_i}}$, $A\subseteq Z_{\alpha}\cap Z_{\beta_i\alpha^{k_i}}$  and  $A\subseteq r(A,\alpha)$.

Let $z\in A\subseteq Z_\alpha\cap Z_{\beta_i\alpha^{k_i}}$. Then, an initial segment of $z$ is $\alpha$ and another is $\beta_i \alpha$. Hence, $\beta_i \alpha=\alpha \gamma$ for some block $\gamma$. Notice that $|\gamma|=|\beta_i|$ and that $\gamma$ is a final segment of $\alpha$. Since $\beta_i$ is also a final segment of $\alpha$ then $\gamma=\beta_i$. Therefore, $\beta_i \alpha=\alpha\beta_i$ and, from the first item of Lemma~\ref{gcd1}, there are natural numbers $m_i,n_i$ such that $\alpha=c_i^{n_i}$ and $\beta_i=c_i^{m_i}$.
In particular, $\alpha={c_3}^{n_3}$ and  $\beta_i\alpha^{k_i}=c_i^{m_i+k_in_i}$ for each $i\in \{3,...,k\}$. Therefore, 
$$x'xy'=\gamma_1 p_A +\gamma_2 s_{c_3^{n_3}} p_A +\sum\limits_{i=3}^k \gamma_i s_{c_i^{m_i+k_in_i}}p_A.$$
Using that $\alpha=c_i^{n_i}=c_j^{n_j}$ for each $i,j\in \{3,...,k\}$ and the third item of Lemma~\ref{gcd1}, we obtain a block $c$ and natural numbers $q_3,...,q_k$ 
such that $c_i=c^{q_i}$ for each $i\in \{3,...,k\}$. Therefore, 
$$x'xy'=\gamma_1 p_A +\gamma_2 s_{c^{q_3n_3}} p_A +\sum\limits_{i=3}^k \gamma_i s_{c^{q_i(m_i+k_in_i)}}p_A.$$  Define $p_2:=q_3n_3$ and $p_i:=q_i(m_i+k_in_i)$, for $i\in \{3,...,k\}$. Then, 
$$x'xy'=\gamma_1 p_A +\sum\limits_{i=2}^k \gamma_i s_{c^{p_i}}p_A.$$
Observe that $c^{p_2}=c^{q_3n_3}=(c^{q_3})^{n_3}=c_3^{n_3}=\alpha$, which implies that $A\subseteq r(A,\alpha)=r(A,c^{p_2})$, $A\subseteq F_{\alpha}=F_{c^{p_2}}$, and $A\subseteq Z_{\alpha}=Z_{c^{p_2}}$. Moreover, for $i\in \{3,...,k\}$, we have that $$\beta_i\alpha^{k_i}=c_i^{m_i+k_in_i}=(c^{q_i})^{m_i+k_in_i}=c^{p_i},$$ so that $A\subseteq F_{\beta_i\alpha^{k_i}}=F_{c^{p_i}}$ and $A\subseteq Z_{\beta_i\alpha^{k_i}}=Z_{c^{p_i}}$.

\smallskip
\noindent {\bf Claim 8: }
{\it Let $c\in \lang$ and $A\in \TCB$. Suppose that $$0\neq x =\gamma_1p_A+\sum\limits_{i=2}^k\gamma_i s_{c^{p_i}}p_A,$$ with $A\subseteq Z_{c^{p_k}}$, and $0\neq \gamma_i\in R$. If $A\nsubseteq \bigcap \limits_{n=1}^\infty Z_{c^n}$, then there exists $x''\in \ualgshift$ such that $0\neq x''x=\gamma_1 p_D$, where $D\in \TCB$.}
\smallskip

Since $A\nsubseteq \bigcap\limits_{n=1}^\infty Z_{c^n}$, there exists $n\in \N$ such that $A\nsubseteq Z_{c^n}$. Notice that $n>p_k$, because $A\subseteq Z_{c^{p_k}}$. Let $m$ be the least of the $i\in \N$ such that $A\nsubseteq Z_{c^{i+1}}$. Observe that $m\geq p_k$, since $A\subseteq Z_{c^j}$ for each $1\leq j\leq p_k$, and $A\subseteq Z_{c^m}$ but $A\nsubseteq Z_{c^{m+1}}$. Let $z\in A$ be such that $z\notin Z_{c^{m+1}}$. Then $z=c^m\beta z'$, where $|\beta|=|c|$ but $\beta\neq c$ (since $z\in Z_{c^{m}}$ and $z\notin Z_{c^{m+1}}$).
Hence, $s_{c^m\beta} s_{c^m\beta}^*p_A=p_{Z_{c^m\beta}}p_A$
 and $p_{Z_{c^m\beta}}p_A\neq 0$, since $z\in A\cap Z_{c^m\beta}$.
Moreover, for each $i\in \{2,...,k\}$, we have that
\begin{align*}
s_{c^m\beta}^*s_{c^{p_i}}p_A&=s_{c^{m-p_i}\beta}^*p_{F_{c^{p_i}}}p_A\\
&=s_{c^{m-p_i}\beta}^*s_{c^{m-p_i}\beta}s_{c^{m-p_i}\beta}^*p_Ap_{F_{c^{p_i}}}\\
&=s_{c^{m-p_i}\beta}^*p_{Z_{c^{m-p_i}\beta}}p_Ap_{F_{c^{p_i}}}.
\end{align*}
Since $A\subseteq Z_m$, $p_i\geq 1$, $|\beta|=|c|$, and $\beta \neq c$, we obtain that $A\cap Z_{c^{m-p_i}\beta}=\emptyset$. Thus, for each $i\in \{2,...,k\}$, we have that
$s_{c^m\beta}^* s_{c^{p_i}}p_A=0$. Consequently, $s_{c^m\beta}s_{c^m\beta}^* x=\gamma_1 p_D\neq 0,$ where $D=Z_{c^m\beta}\cap A$, and Claim~8 is proved. 

\smallskip

\noindent{\bf Proof of the Theorem:}

Let $0\neq x\in \ualgshift$. By Claim 7, there exist $x',y'\in \ualgshift$ such that $0\neq x'xy'$ and  either $x'xy'=\gamma p_D$ with $\gamma \in R$ and $D\in \TCB$, or
$$x'xy'=\gamma_1 p_A+\sum\limits_{i=2}^k\gamma_i s_{c^{p_i}}p_A,$$ where $c\in \lang$, $0\neq \gamma_i\in R$ for each $i\in \{1,...,k\}$, $1\leq p_i <p_{i+1}$, $A\subseteq F_{c^{p_i}}\cap Z_{c^{p_i}}$, and  $A\subseteq r(A,c^{p_2})$.
If $A\nsubseteq \bigcap \limits_{n=1}^\infty Z_{c^n}$ then, by Claim 8, there exists $x''\in \ualgshift$ such that $x'' (x'xy')=\gamma p_{\Tilde{D}}\neq 0$, where $\Tilde{D}\in \TCB$. In the case $A\subseteq \bigcap\limits_{n=1}^\infty Z_{c^n}$, we have that $A=\{c^\infty\}$, which implies that $(A,c)$ is a cycle without exit. By Remark \ref{minimal cycle}, there exists $\beta\in \lang$ such that $c=\beta^l$ for some $l\geq 2$, and such that $(A,\beta)$ has minimal length. Then, as $c=\beta^l$ and $$x'xy'=\gamma_1 p_A+\sum\limits_{i=2}^k\gamma_i s_{c^{p_i}}p_A,$$ we have that 
$$x'xy'=\gamma_1 p_A+\sum\limits_{i=2}^k\gamma_i s_{\beta^{l p_i}}p_A,$$ which  finishes the proof.
\end{proof}

\section{Some consequences of the Reduction Theorem}\label{consequencesofreduction}

Let $R$ be a commutative ring and $\osf$ be a subshift. In this section, we provide some results related to the subshift algebra $\ualgshift$ that emerge as a consequence of the Reduction Theorem. Specifically, we prove a Cuntz-Krieger uniqueness theorem for $\ualgshift$ and give some ring-theoretic properties of $\ualgshift$.

\begin{theorem}\label{ckuniqueness}{\rm {\bf (Cuntz-Krieger uniqueness theorem)}} Let $\osf$ be a subshift such that each cycle has an exit, $R$ be a commutative ring, and $\ualgshift$ be the subshift algebra. If $B$ is an $R$-algebra and $\Phi:\ualgshift\rightarrow B$ is an $R$-homomorphism such that $\Phi(\gamma p_A)\neq 0$ for each $\emptyset \neq A\in \TCB$ and $0\neq \gamma \in R$,  then $\Phi$ is injective.
\end{theorem}

\begin{proof} Suppose that $\operatorname{ker}\Phi\neq 0$ and let $0\neq x\in \operatorname{ker}\Phi$. Since each cycle has an exit, by Theorem \ref{reduction theorem}, there exist $x',y'\in \ualgshift$ such that $0\neq x'xy'=\gamma p_A$, where $0\neq \gamma\in R$ and $ A\in \TCB$, $A \neq \emptyset$. Then, $0=\Phi(x')\Phi(x)\Phi(y')=\Phi(\gamma p_A)$, which contradicts the hypothesis. Therefore, $\operatorname{ker}\Phi=0$ and $\Phi$ is injective.
\end{proof}

Our next goal is to show that the subshift algebra is semiprimitive. For this, we will need a couple of lemmas. 
In the next one, we explore the structure of a corner of a subshift algebra by a projection over a minimal cycle without exit. In what follows, for $n<0$, we use the notation $s_c^n$ meaning the element $(s_c^*)^{-n}$.

\begin{lemma}\label{cornerlemma} Let $\osf$ be a subshift and let $(A,c)$ be a minimal cycle without exit. Then, $$p_A \ualgshift p_A=\left\{\sum\limits_{n=i}^j\gamma_n s_c^np_A: \gamma_n\in R \text{ and }i,j\in \Z\right\}.$$
\end{lemma}

\begin{proof} Recall that $\ualgshift$ is the linear span of elements of the form $s_\alpha p_Bs_{\beta}^*$, where $B\in \TCB$ and $\alpha,\beta\in \lang$ (including the cases $|\alpha|=0$ or $|\beta|=0$). In what follows, we analyze an element $p_As_\alpha p_B s_\beta^* p_A$. 

\smallskip
\noindent{\it {\bf Claim}: Let $\alpha,\beta\in \lang$ and $B\in \TCB$. Then, $p_A s_\alpha p_B s_\beta^*p_A=0$ or $p_A s_\alpha p_B s_\beta^*p_A=p_A s_\alpha s_\beta^*p_A$.}
\smallskip

Indeed, notice that $$p_As_\alpha p_B s_\beta^* p_A=s_\alpha p_{r(A,\alpha)}p_Bp_{r(A,\beta)}s_\beta^*=s_\alpha p_{r(A,\alpha)\cap B\cap r(A,\beta)}s_\beta^*.$$
Since $A$ is a singleton set, we have that $r(A,\alpha)$ is also a singleton set. Hence, $r(A,\alpha)\cap B=\emptyset$ or $r(A,\alpha)\cap B=r(A,\alpha)$. Therefore, $r(A,\alpha)\cap B \cap r(A,\beta)=\emptyset$ or $r(A,\alpha)\cap B \cap r(A,\beta)=r(A,\alpha)\cap r(A,\beta)$. From this, we conclude that $p_As_\alpha p_B s_\beta^*=0$ or 
$$p_As_\alpha p_B s_\beta^*=s_\alpha p_{r(A,\alpha)\cap r(A,\beta)}s_\beta^*=p_A s_\alpha s_\beta^* p_A,$$ where the last equality follows by Lemma~\ref{chucrute}(2).

By the previous claim, $\ualgshift$ is generated by elements of the form $p_As_\alpha s_\beta^* p_A$, which we analyze next.

For $\alpha \in \lang$, notice that $p_As_\alpha=p_A s_\alpha s_\alpha^*s_\alpha=p_A p_{Z_\alpha}s_\alpha$. Since $A=\{c^\infty\}$, we have that $p_A p_{Z_\alpha}=0$ (and hence $p_As_\alpha=0$), unless $\alpha=c^p a$, where $a$ is an initial segment of $c$ and $p\in \N$. Similarly, $s_\beta^*p_A=0$, unless $\beta=c^qb$, where $q\in \N$ and $b$ is an initial segment of $c$. Hence, in what follows in this proof, $\alpha$ and $\beta$ are elements of $\lang$ such that $|\alpha|>0$, $|\beta|> 0$, and $\alpha=c^p a$, $\beta=c^q b$, where $p,q\in \N$ and $a$, $b$ are initial segments of $c$ such that $0\leq |a|<|c|$ and $0\leq |b|<|c|$.
Observe that
$$p_A s_\alpha p_A=p_A s_{\alpha}s_{\alpha}^*s_\alpha p_A=p_A s_{\alpha}p_{r(A,\alpha)}p_A=p_A s_\alpha p_{r(A,\alpha)\cap A}.$$ If  $r(A,\alpha)\cap A=\emptyset$ then $p_A s_\alpha p_A=0$. Otherwise, $r(A,\alpha)\cap A=\{c^\infty\}$ (because $A=\{c^\infty\}$). Then, $c^\infty\in r(A,\alpha)$ and hence $\alpha c^\infty =c^\infty$. Since $\alpha =c^p a$, the previous equality implies that $c^p ac^\infty=c^\infty$ and consequently $ac^\infty=c^\infty$. By Lemma~\ref{minimalversion01}, since $(A,c)$ is a minimal cycle, we conclude that $|a|=0$. Thus, $\alpha=c^p$ and $p_As_\alpha p_A=p_A s_c^p p_A$. Similarly, we have that $p_A s_\beta^* p_A=0$ or $p_A (s_c^*)^q p_A$. 

Next, we analyze the case $p_As_\alpha s_\beta^* p_A$. It follows from Lemma~\ref{chucrute}(2) that 
$$p_As_\alpha s_\beta^* p_A=s_\alpha p_{r(A,\alpha)\cap r(A,\beta)}s_\beta^*.$$ Since $A=\{c^\infty\}$, we have that either $r(A,\alpha)\cap r(A,\beta)=\emptyset$,  or $r(A,\alpha)\cap r(A,\beta)$ is a singleton set. In the first case, $p_As_\alpha s_\beta^*p_A=0$. To deal with the second case, assume that $r(A,\alpha)\cap r(A,\beta)=\{x\}$. 
Then, $\alpha x=c^\infty=\beta x$. As $\alpha =c^p a$, with $0\leq |a|<|c|$, the previous equality implies that 
$$c^p a x=c^\infty \text{ and } ax=c^\infty.$$
Similarly, $bx=c^\infty$. Now, let $a',b'\in \lang$ be such that $aa'=c$ and $bb'=c$. From $ax=c^\infty$, we obtain that $x=a'c^\infty$ and similarly, from $bx=c^\infty$, we have that $x=b'c^\infty$. Thus,
$$c^\infty=bx=ba' c^\infty, \text{ and } c^\infty=ax=ab'c^\infty.$$
Notice that $|a|+|a'|=|c|=|b|+|b'|$, which implies that $|ab'|+|ba'|=2|c|$. Hence, $|ab'|\leq |c|$ or $|ba'|\leq |c|$. We suppose, without loss of generality, that $|ab'|\leq |c|$. We have two possibilities. If $|ab'|<|c|$ then, by Lemma~\ref{minimalversion01}, we obtain that $|ab'|=0$ (since $ab'c^\infty=c^\infty$). Thus, $|b'|=0$ and consequently $|b|=|c|$, which is impossible, because $|b|<|c|$. Hence $|ab'|=|c|$. In this case, $$|a|+|b'|=|ab'|=|c|=|a|+|a'|,$$ 
which implies that $|a'|=|b'|$. Therefore, $|a|=|b|$ and consequently $a=b$. So, we have that  $\alpha=c^pb$, $\beta=c^qb$, and 
$$p_A s_\alpha s_\beta^*p_A=p_A s_c^ps_bs_b^*(s_c^*)^qp_A=p_A s_c^pp_{Z_b}(s_c^*)^qp_A.$$
From the claim at the beginning of this proof, we obtain that  
$$p_A s_\alpha s_\beta^*p_A=p_A s_c^pp_{Z_b}(s_c^*)^qp_A=0 \,\,\text{ or }\,\, p_A s_\alpha s_\beta^*p_A=p_A s_c^pp_{Z_b}(s_c^*)^qp_A=p_A s_c^p(s_c^*)^qp_A.$$ 
Suppose that $p_A s_\alpha s_\beta^*p_A=p_A s_c^p(s_c^*)^qp_A$ and that $p>q$. Write $p=q+n$ and note that 
\begin{align*}
    p_A s_\alpha s_\beta^*p_A&=p_A s_c^p(s_c^*)^qp_A=p_A s_c^ns_c^q (s_c^*)^qp_A\\
    &=p_A s_c^np_{Z_{c^q}}p_A=p_As_c^np_A.
\end{align*}
In a similar way, if $q>p$ then we write $q=p+m$ and $p_A s_\alpha s_\beta^*p_A=p_A(s_c^*)^mp_A$. For $p=q$, we have that $p_A s_\alpha s_\beta^*p_A=p_A$. \smallbreak

Finally, we will show that, for each $n\in \Z$, the equality $p_A s_c^np_A=s_c^np_A$ is true. When $n=0$, we have that $p_A s_c^np_A=p_A$. So, assume that $n\neq 0$. Since $A=\{c^\infty\}$, we obtain that $A=r(A,c^{|n|})\cap A=r(A, c^{|n|})$, for each $n\in \Z$. Thus, in the case that $n<0$, we get from the second item of Lemma~\ref{chucrute} that 
\begin{align*}
    p_A s_c^n p_A&=p_A (s_c^*)^{|n|}p_A= p_A p_{r(A,c^{|n|})}(s_c^*)^{|n|}\\
    &=p_{r(A, c^{|n|})}(s_c^*)^{|n|}=(s_c^*)^{|n|}p_A=s_c^n p_A.
\end{align*}
For $n>0$, we also have that $$p_A s_c^n p_A\stackrel{(\star)}{=}s_c^n p_{r(A,c^n)}p_A\stackrel{(\star\star)}{=}s_c^n p_A,$$ where $(\star)$ follows from Lemma~\ref{chucrute}(2) and $(\star\star)$ follows from $r(A,c^n) \cap A=A$.
\end{proof}

As a consequence of Lemma~\ref{cornerlemma}, we can identify the corner $p_A\ualgshift p_A$ with the Laurent polynomials algebra, see below.

\begin{lemma}\label{lem-Laurent-pol} Let $\osf$ be a subshift and $(A,c)$ be a minimal cycle without exit. Then,  $$\psi:p_A\ualgshift p_A\to R[x,x^{-1}]$$ given by
\[\psi(s^n_cp_A)=x^n, \quad \text{for all } \,n\in \Z,\]
is an $R$-algebra isomorphism.
\end{lemma}
\begin{proof} It follows from Lemma \ref{cornerlemma} that each $y\in p_A \ualgshift p_A$ has the form $y=\sum\limits_{n=i}^j\gamma_n s_c^n p_A$, where $0\neq \gamma_n\in R$, and $i,j\in \Z$. By Theorem \ref{thm:set-theoretic-partial-action},  we have that $\gamma_ns_c^n p_A\neq 0$ for each $n\in\Z$. Moreover, $y$ can be uniquely written as $y=\sum\limits_{n=i}^j\gamma_n s_c^n p_A$ (with $0\neq \gamma_n \in  R$).
So, $\psi:p_A \ualgshift p_A \rightarrow R[x,x^{-1}]$ given by $\psi(y)=\sum\limits_{n=i}^j \gamma_n x^n$ is a well-defined linear map. 
At the end of the proof of Lemma~\ref{cornerlemma}, we showed that $p_A s_c^n p_A=s_c^n p_A$. Therefore, for $m,n\in \Z$, we have that $$(s_c^np_A)(s_c^m p_A)=s_c^ns_c^mp_A.$$
Furthermore, observe that $s_c^ns_c^m p_A=s_c^{n+m}p_A$ for each $n,m\in \Z$. 
Hence, we have that $$\psi(s_c^np_A)\psi(s_c^mp_A)=x^nx^m=x^{n+m}=\psi(s_c^{n+m}p_A)=\psi((s_c^np_A)(s_c^mp_A)),$$ for each $n,m\in \Z$. 
Thus, $\psi$ is a homomorphism of $R$-algebras.
Now, using the universal property of $R[x,x^{-1}]$, there exists a homomorphism $\phi:R[x,x^{-1}]\rightarrow p_A \ualgshift p_A$ such that $\phi(x^n)=s_c^np_A$, for all $n\in \Z$. Since $\phi=\psi^{-1}$, the result is proved.
\end{proof}

In order to investigate ring-theoretic properties of $\ualgshift$, we recall the following notions. A ring $S$ is said to be \emph{semiprime} if whenever $I^2=0$ for an ideal $I$ of $S$, then $I=0$. 
 The \emph{Jacobson radical} $J(A)$ of an algebra $A$ is the intersection of all primitive ideals in $A$, whereas the prime radical $\operatorname{rad}(A)$ is the intersection of all prime ideals in $A$. Furthermore, a ring $S$ is called \emph{semiprimitive} if the Jacobson radical $J(S)$ of $S$ is zero. Next, we show that the subshift algebra over a field is semiprimitive and semiprime.

\begin{proposition}\label{ecologico} Let $\osf$ be a subshift and $R$ be a field. Then, $\ualgshift$ is semiprimitive.
 \end{proposition}
\begin{proof}
Let $J=J(\ualgshift)$ be the Jacobson radical of $\ualgshift$ and suppose that there exists $0\neq x\in J$.  By 
Theorem~\ref{reduction theorem}, there exist $\mu,\nu\in \ualgshift$ such that  $0\neq \mu x\nu$ and either $\mu x\nu=\gamma p_D$, with $0\neq \gamma\in R$ and $D\in \TCB$, or 
 \[    \mu x\nu=\gamma_1 p_A +\sum\limits_{i=2}^k \gamma_i s_{\beta^{q_i}}p_A,
\]
where $(A,\beta)$ is a minimal cycle without exit, $q_i\in \N\setminus \{0\}$ and $0\neq \gamma_i\in R$. Observe that if $\mu x\nu=\gamma p_D$, then  
$p_D=\gamma^{-1}(\mu x\nu)\in J$. Since $J$ contains no nonzero idempotents, we have a contradiction. Hence, $\mu x\nu$ is given by the sum above. By Lemma~\ref{cornerlemma} we have that $\mu x\nu$ is a nonzero element in $J\cap p_A\ualgshift p_A$. By \cite[\S III.7, Proposition 1]{Jac64}, $J\cap p_A\ualgshift p_A=J\big(p_A\ualgshift p_A\big)$. Using Lemma \ref{lem-Laurent-pol}, we obtain a nonzero element in $J(R[x,x^{-1}])$, which is again a contradiction. Thus, $J=0$ and the result follows.
\end{proof}

\begin{remark}\label{jacare}
We recall that an algebra $A$ is semiprime whenever $\operatorname{rad}(A)=0$. Since every primitive ideal of $A$ is a prime ideal of $A$, it follows that if $A$ is semiprimitive then $A$ is semiprime. Hence, by the previous Proposition \ref{ecologico}, if  $\osf$ is a subshift and $R$ is a field then $\ualgshift$ is semiprime.
\end{remark}

We have described in the remark above how to show that $\ualgshift$ is semiprime, provided that $\osf$ is a subshift and $R$ is a field. However, we can improve this result as follows.

\begin{corollary}\label{banana} Let $\osf$ be a subshift and $R$ be a ring without zero divisors. Then, $\ualgshift$ is semiprime.
\end{corollary}

\begin{proof} Let $I$ be a nonzero ideal of $\ualgshift$. Applying Theorem~\ref{reduction theorem} to a nonzero element of $I$, we obtain a nonzero element $x\in I$ of the form $x=\gamma p_D$, with $\gamma\in R$ and $D\in \TCB$, or 
$x=\gamma_1 p_A +\sum\limits_{i=2}^k \gamma_i s_{\beta^{q_i}}p_A,$
where $(A,\beta)$ is a minimal cycle without exit, $q_i\in \N\setminus \{0\}$, and $0\neq \gamma_i\in R$. Using the grading given by Theorem~\ref{thm:set-theoretic-partial-action} and from the fact that $R$ has no zero divisors, we get that $x^2\neq 0$ and therefore $I^2\neq 0$, as desired.
\end{proof}

{\bf Conflict of interest statement:}
All authors declare that they have no conflicts of interest. 

{\bf Data availability statement:}  The authors confirm that the data supporting the findings of this study are available within the article.

\bibliographystyle{abbrv}
\bibliography{ref}

\end{document}